\documentclass{svjour3}
\smartqed
\usepackage{amscd,amsmath,amsfonts}
\usepackage[all]{xypic}

\newcommand{\dddet}{\underline{\underline{\det}}}
\newcommand{\ddet}{\underline{\det}}
\newcommand{\san}{\mathrm{S}_A^n}
\newcommand{\End}{\mathrm{End}}

\newcommand{\Hilb}{\mathrm{Hilb}}
\newcommand{\GL}{\mathrm{GL}}
\newcommand{\TS}{\mathrm{TS}}
\newcommand{\MP}{\mathrm{MP}}
\newcommand{\PP}{\mathrm{P}}
\newcommand{\NN}{\mathbb{N}}
\newcommand{\Z}{\mathbb Z}

\newcommand{\al}{\alpha}
\newcommand{\bt}{\beta}

\newcommand{\A}{\mathcal{A}}

\newcommand{\C}{\mathcal{C}}

\newcommand{\nn}{\mathcal{N}}

\newcommand{\pd}[2]{#1^{(#2)}}

\newcommand{\CC}{\mathbb{C}}

\newcommand{\ZZ}{\mathbb{Z}}

\newcommand{\ran}{\mathrm{Rep}_A^n}
\newcommand{\rana}{\ran \times _R \mathbb{A}^n_R}

\newcommand{\uan}{\mathrm{U}_A^n}
\newcommand{\van}{V_n(A)}

\begin{document}

\title{Hilbert-Chow morphism for non commutative Hilbert schemes and moduli spaces of linear representations\thanks{The first author is supported by Progetto di Ricerca Nazionale COFIN 2006 "Geometria delle Variet\`a
Algebriche e dei loro Spazi di Moduli".}
}

\titlerunning{Non commutative Hilbert-Chow}

\author{Federica Galluzzi         \and
        Francesco Vaccarino
}

\authorrunning{F.Galluzzi, F.Vaccarino}

\institute{F. Galluzzi \at
              Dipartimento di Matematica, Universit\`a di Torino, Via Carlo
Alberto 10, 10123 Torino, ITALY \\
              \email{federica.galluzzi@unito.it}                      \and
           F. Vaccarino \at
Dipartimento di Matematica, Politecnico di Torino, C.so Duca
degli Abruzzi 24, 10129 Torino, ITALY\\
\email{francesco.vaccarino@polito.it}
}

\date{}

\maketitle

\begin{abstract}
Let $k$ be a commutative ring and let $R$ be a commutative $k-$algebra. The aim of this paper is to define and discuss some connection morphisms between
schemes associated to the representation theory of a (non necessarily commutative) $R-$algebra $A.\,$ We focus on the scheme $\ran//\GL_n$ of the
$n-$dimensional representations of $A,\,$ on the Hilbert scheme
$\Hilb_A^n$ parameterizing the left ideals of codimension $n$ of $A$ and
on the affine scheme Spec $\Gamma_R^n(A)^{ab} \,$ of the abelianization
of the divided powers of order $n$ over $A.\,$

We give a generalization of the Grothendieck-Deligne norm map from $\Hilb_A^n$ to Spec $\Gamma_R^n(A)^{ab} \,$ which
specializes to the Hilbert Chow morphism on the geometric points when $A$ is
commutative and $k$ is an algebraically closed field. Describing the Hilbert scheme as the base of a principal bundle we shall factor this map through the moduli space $\ran//\GL_n$ giving a nice description of this Hilbert-Chow morphism, and consequently proving that it is projective.
\end{abstract}

\section*{Introduction}
Let  $k$ be an algebraically closed field and let  $A$ be a commutative $k-$algebra with $X=\mathrm{Spec}\,A$.
The $k-$points of the Hilbert scheme $\Hilb_A^n$ of $n-$points on $X$ parameterize zero-dimensional
subschemes $Y\subset X$ of length $n.\,$ It is the simplest case of Hilbert scheme
parameterizing closed subschemes of $X$ with fixed Hilbert polynomial $P,\,$  in this case $P$ is the constant polynomial $n.\,$ see \cite{BK,LST}.
The $n-$fold symmetric product $X^n/S_n$ of $X$ is defined as Spec $(A^{\otimes n})^{S_n}$, where the symmetric group $S_n$ acts naturally on $A^{\otimes n}$. It parameterizes the effective $0-$cycles of length $n$ on $X$, (see \cite{BK}).
There is a natural set-theoretic map
\[
\Hilb^n_A \to X^n/S_n
\]
mapping a zero-dimensional subscheme $Z $  in $\Hilb^n_A$ to the $0-$cycle of degree $n\,$:
\[
\displaystyle{ [Z]\, =\,\sum _{P \in |Z|}\, \dim _k(\mathcal O_{Z,P})[P]}.
\]
This is indeed a morphism of schemes, the \textit {Hilbert-Chow morphism} (see for example \cite{BK,Fo,Iv,Ko} ). If $X$ is a non singular curve, this map is an isomorphism (see \cite{De,Gr}). If $X$ is a non singular surface, the Hilbert-Chow morphism is a resolution of singularities of $X^n/S_n\,$ (see \cite{Fo}), but this is no longer true in higher dimensions.

The aim of this paper is to define and study a generalization of this setting according to the following direction: the Hilbert scheme is the representing scheme of a functor from commutative rings to sets which can be easily extended to a functor from (non commutative) algebras to sets. This latter  is given by
\[
\begin{array}{ll}
\Hilb_A^n(B) := & \{\mbox {left ideals } I
\mbox { in } A \otimes_{R} B \mbox { such that } M=A \otimes _R B /I \
\mbox {is projective} \\
& \mbox{ of rank } n
\mbox { as a } B\mbox{-module} \}.
\end{array}
\]
where $R$ is a commutative  ring,  $A,\,B$ are $R-$algebras with $B$ commutative. Since this is a closed subfunctor of the Grassmannian functor it is clearly representable. It has been proved (in a special but fundamental case) by M.Nori \cite{No} that this functor is represented by a scheme, which turns out to be a principal $\GL_n-$bundle.  Following M.Reineke \cite{Re} we call this scheme {\textit{the non commutative Hilbert scheme}}. It should be evident that it coincides with the usual Hilbert scheme of $n-$points when $A$ is commutative.

The description of $\Hilb_A^n$ as a principal $\GL_n-$bundle lands us to the Geometric Invariant Theory of linear representations of algebras as defined by M.Artin \cite{A} and studied by S.Donkin \cite{Do}, C.Procesi \cite{Pr2,Pr3} and A.Zubkov \cite{Zu}.  In particular we have a description of $\Hilb_A^n$ as the quotient by $\GL_n$ of an open subscheme $\uan$ of $\rana$ where $\ran$ is the scheme representing the $n-$dimensional linear representations of $A$. This gives us a morphism $\Hilb_A^n\to \ran//\GL_n$ which turns out to be projective.

Let us anticipate how this work in a fundamental case: suppose $A$ is the free associative ring on $m$ variables. Its $n-$dimensional
linear representations are in bijection with the affine space of $m-$tuples of $n\times n$ matrices. On this space it acts the general linear group by simultaneous conjugation i.e. via bases change. The coarse moduli space parameterizing the orbits for this action is the spectrum of the invariants on $m-$tuples of matrices. Using results due to V.Balaji \cite{Ba} we are able to characterize the Hilbert scheme as the quotient of an open subscheme of $m-$tuples of matrices times the standard representation of the general linear group. The morphism $\Hilb_A^n\to \ran//\GL_n$ is then given by passing to the quotients the projection $\rana\to\ran$ i.e. on geometric point the orbit of $(a_1,\dots,a_m,v)$ maps to the one of $(a_1,\dots,a_m)$. One {\it forgets} the vector.

How to insert the symmetric product into this framework?
This the idea we pursued: suppose $A$ is free as $R-$module, composing a representation over $B$ with the determinant we generate a multiplicative polynomial mapping homogeneous of degree $n$. This corresponds to a unique morphism $(A^{\otimes n})^{S_n}\to B$ i.e. to a $B-$point of Spec $(A^{\otimes n})^{S_n}/<\{xy-yx\,:\,x,y\in (A^{\otimes n})^{S_n}\}>$. This seems to be a good candidate for the codomain of the Hilbert-Chow morphism. The only replacement to be done to overcome the freeness requirement is to substitute $(A^{\otimes n})^{S_n}$ with $\Gamma^n(A)$, where $\Gamma(A)=\bigoplus_n\Gamma^n(A)$ is the divided power algebra on $A$. In fact for an $R-$algebra $A$ one has that $\Gamma^n(A)$ represents the functor
\[B\to \{\mbox {multiplicative polynomial mapping homogeneous of degree $n$ from $A$ to $B$}\},\]
where $B$ is a commutative $R-$algebra and a {\it polynomial law} is a generalization of the concept of polynomial mapping which will be explained in the article.

In this way we have \[\Hilb_A^n\to \ran//\GL_n\to \mathrm{Spec}\,\Gamma^n(A)/<\{xy-yx\,:\,x,y\in \Gamma^n(A)\}>.\]
When $A$ is commutative this composition coincides with the Grothendieck-Deligne norm map introduced in \cite{LST} and further studied in \cite{ES}. Henceforth with the Hilbert-Chow morphism when $k$ is algebraically closed morphism. We call this map \textit{the non commutative Hilbert-Chow morphism}.
Thanks to this description we shall easily prove that the Hilbert-Chow morphism is always projective.
Using some result due to the second author we will show that the Hilbert-Chow morphism the same as the projection $\Hilb_A^n\to \ran//\GL_n$ in some important and fairly general cases.

The paper goes as follows. In Section\,\ref{funct} we recall the definition of the coarse moduli space  $\ran//\GL_n$ of the $n-$dimensional representations of $A,\,$ using the algebra of the generic matrices. Then we define an open subscheme $\uan $ in $\rana$ and we recall how to construct a principal $\GL_n-$bundle $\uan\to\uan/\GL_n .\,$ In Section\ref{nchs} we survey the definition and the main properties of a representable functor of points whose representing scheme $\Hilb ^n _A$ is the usual Hilbert scheme of points when $A$ is commutative.
In \ref{bund} we prove that the scheme $\uan/\GL_n$ represents $\Hilb_A^n$ and that $\uan\to \Hilb_A^n$ is a universal categorical quotient and a principal $\GL_n-$bundle.

In Section\,\ref{poly} we introduce polynomial laws and the functor $\Gamma^n(A)$ of the divided powers over $A.\,$ When $A$ is flat $\Gamma^n(A)$ is isomorphic to the symmetric tensor functor and will play the same role played by the symmetric product in the classical Hilbert-Chow morphism.

Section \ref{mor} is devoted to define morphisms which connect the schemes introduced
in the previous sections. In \ref{proj} we prove that the morphism $p: Hilb ^n _A \to \ran//\GL_n$
induced by the "forgetful map" is projective.

Then we define a norm map $hc:\Hilb^n _A\to \mathrm{Spec}\,\Gamma^n(A)^{ab}$ which generalizes the Grothendieck-Deligne norm map and specializes to the classical Hilbert-Chow morphism in the commutative case when the ground field is algebraically closed.


\section{Notations}\label{not} Unless otherwise stated we adopt the following
notations:
\begin{itemize}
\item $k$ is a fixed commutative ground ring.
\item $R$ is a commutative $k-$algebra.
\item $B$ is a commutative $R-$algebra.
\item $A$ is a not necessarily commutative $R-$algebra.
\item $F=k\{x_1,\dots,x_m\}$ denotes the associative free $k-$algebra on
$m$ letters.
\item $\mathcal{N}_-,\,\mathcal{C}_-, Mod_-$ and $Sets$ denote the
categories of $-$algebras,
commutative $-$algebras, $-$modules and sets, respectively.
\item we write $\A(B,C):=Hom_{\A}(B,C)$ in a category $\A$ with
$B,C$ objects in $\A$.
\end{itemize}

\section{The moduli space of representations}\label{funct}

\subsection{The universal representation }
We denotes by $M_n(B)$ the full ring of $n \times n$ matrices over
$B.\,$ If $f \ : \ B \to C $ is a ring homomorphism we
denote with
$
M_n(f)\ : \ M_n(B) \to M_n(C)
$
the homomorphism induced on matrices.
\begin{definition}
By an { \em n-dimensional representation of} $A$ over $B$ we mean a
homomorphism of $R-$algebras $
\rho\, : A \, \to M_n(B).\,
$\end{definition}
The assignment $B\to \nn_R(A,M_n(B))$ defines a covariant functor
$\mathcal {N}_{R} \rightarrow Sets .\,$
This functor is represented by a commutative $R-$algebra. We report here
the proof of this fact to show how this algebra comes up using generic matrices.
These objects will be also crucial in the construction of the norm map
in Section \ref{nmap}.

\begin{lemma}\cite[Lemma 1.2.]{DPRR}\label{repr}
For all $A\in\nn_R$ there exist a commutative $R-$algebra $V_n(A)$ and a
representation $\pi_A:A\to M_n(V_n(A))$ such that $\rho\mapsto
M_n(\rho)\cdot\pi_A$ gives an isomorphism
\begin{equation}\label{proc}
\C_R(V_n(A),B)\xrightarrow{\cong}\nn_R(A,M_n(B))\end{equation}
for all $B\in C_R$.
\end{lemma}
\begin{proof}
Suppose $A=R\{x_1,\dots,x_m\}=R\otimes_k F$ is the free associative $R-$algebra on $m$ letters. Let $V_n(A)=R[\xi_{kij}]$ be the polynomial ring in variables $\{\xi_{kij}\,:\, i,j=1,\dots,n,\, k=1,\dots,m\}$ over the base ring $R$. To every $n-$dimensional representation of $A$ over $B$ it corresponds a unique $m-$tuple of $n\times n$ matrices, namely the images of $x_1,\dots,x_m$, hence a unique $\bar{\rho}\in\C_R(R[\xi_{kij}],B)$ such that $\bar{\rho}(\xi_{kij})=(\rho(x_k))_{ij}$.
Following C.Procesi {\cite{Pr1,DPRR}} we introduce the generic matrices.
Let $\xi_k=(\xi_{kij})$ be the $n\times n$ matrix whose $(i,j)$ entry is $\xi_{kij}$ for $i,j=1,\dots,n$ and $k=1,\dots,m$.  We call   $\xi_{1},\dots,\xi_{m}$ the generic $n\times n$ matrices.
Consider the map
\[
\pi_A:A \to
M_n(V_n(R)),\;\;\;\;x_{k}\longmapsto \xi_{k},\;\;\; k=1,\dots,m\,.
\]
It is then clear that the map $\C_R(V_n(A),B)\ni \sigma\mapsto M_n(\sigma)\cdot\pi_A\in\nn_R(A,M_n(B))$ gives the isomorphism (\ref{proc}) in this case.

Let now $A=R\otimes_k F/C$ be an associative $R-$algebra and write $\bt:R\otimes_k F\to A$ for the homomorphism such that $C=\ker \bt$.
Suppose $a_{k}=\beta(x_{k})$, for $k=1,\dots,m$.
As before consider $V_n(\ R \otimes_k F)=R[\xi_{kij}]$: an $n-$dimensional representation $\rho$ of $A$ over $B$ lifts to one of $R\otimes_k F$ by composition with $\beta$. This gives a homomorphism $R[\xi_{kij}]\to B$ that factors through the quotient $R[\xi_{kij}]/I$, where $M_n(I)$ is the ideal of $M_n(R[\xi_{kij}])$ generated by $\pi_{R\otimes F}(C)$.

We set $V_n(A)=R[\xi_{kij}]/I$ and $\xi_k^A=(\xi_{kij}+I)=\xi_k+M_n(I)\in M_n(V_n(A))$ for $k=1,\dots,m$. There is then a homomorphism
\[\pi_A:A\to M_n(V_n(A))\]
given by $\pi_A(a_k)=\xi_k^A$ for $k=1,\dots,m$.
To conclude given $\rho\in\nn_R(A,M_n(B))$ there is a unique homomorphism of commutative $R-$algebras
\begin{equation}\label{class}
\bar{\rho} \ : V_n(A) \to B
\end{equation}
for which the following diagram
\begin{equation}\label{univ}
\xymatrix{
A \ar[r]^(.3){\pi _A} \ar[rd]_{\rho}&
M_n(V_n(A))\ar[d]^{M_n(\bar{\rho})} \\
& M_n(B) }
\end{equation}
commutes.
\qed\end{proof}
\begin{remark}
It should be clear that the number of generators $m$ of $A$ is immaterial and we can extend the above isomorphism to the not finitely generated case.
\end{remark}
\begin{definition}\label{pigreco}
We write $\ran$ to denote Spec\,$V_n(A).\,$ It is considered as an
$R-$scheme. \\ The map
\begin{equation}
\pi _A :A \to
M_n(V_n(A)),\;\;\;\;x_{k}\longmapsto \xi_{k}.
\end{equation}
is called {\em the universal n-dimensional representation.}

\smallskip
Given a representation $\rho : A \to M_n(B)$ we denote by $\bar{\rho}$
its classifying map $\bar{\rho}:V_n(A)\to B\,$ (see (\ref{class})).
\end{definition}

\begin{example}\label{fcase}
For the free algebra one has $\ran\cong M_n^m$ the scheme whose $B-$points are the $m-$tuples of $n\times n$ with entries in $B$.
\end{example}

\begin{example}\label{commuting}
Note that $\ran$ could be quite complicated, as an example,
when $A=\CC [x,y]$ we obtain that $\ran$ is the \textit{commuting scheme} i.e. the couples of commuting matrices and it is not even known (but conjecturally true) if it is reduced or not, see \cite{V3}.
\end{example}

\subsection{An open subscheme}
For any $B\in\C_R\,$ identify $B^n$ with $\mathbb{A}^n_R(B),\,$
the
$B-$points of the $n-$dimensional affine scheme over $R.\,$
We introduce another functor that is one of the cornerstones of our construction.

\begin{definition}
For each $B\in \C_R$, let $\uan(B)$ denote the set of $B-$points $(\bar{\rho},v)$ of $\rana$
such that $\rho(A)(Bv)=B^n$, i.e. such that $v$ generates $B^n$ as $A-$module via $\rho:A\to M_n(B)$.
\end{definition}

\begin{remark}
It is easy to check that the assignment $B\mapsto \uan(B)$ is functorial in $B$.
Therefore we get a subfunctor $\uan$ of $\rana$ that is clearly open and we denote by $\uan$ the open subscheme of $\rana$ which represents it.
\end{remark}
\subsection{$\GL_n-$actions}\label{repaction}
\begin{definition}
We denote by $\GL_{n}$ the affine group scheme over $R$.
Then its $B-$points form the group $ \GL_n(B)$ of $n \times n$ invertible matrices with entries in
$B$, for all $B\in\C_R$.
\end{definition}
Define a $\GL_{n}$-action on $\ran$ as follows. For any $\varphi \in \ran
(B),\, g \in \GL_{n}(B),\,$ let $\varphi ^g : \van \to B$ be the $R-$algebra
homomorphism corresponding to the representation given by
\begin{equation}
\begin{matrix}
A & \to & M_n(B) \\
a & \longmapsto & g (M_n(\varphi)\cdot \pi _A(a))g^{-1}.
\end{matrix}
\end{equation}
Note that if $\varphi, \varphi '$ are $B-$points of $\ran ,\,$then the
$A-$module structures induced on $B^n$ by $\varphi$ and $\varphi '$ are
isomorphic if and only if there exists $g \in \GL_{n}(B)$ such that $\varphi
' = \varphi ^g .$

\begin{definition}\label{modrep}
We denote by $\ran//\GL_n=\mathrm{Spec}\,V_n(A)^{\GL_{n}(R)}\,$ the
categorical quotient (in the category of $R-$schemes) of $\ran$ by
$\GL_{n}.\,$ It is the \textit{(coarse) moduli space of $n-$dimensional
linear representations} of $A$.
\end{definition}

We can define an action of $\GL_n\,$ on
$\rana $ similar to the one on $\ran$, namely for any $B \in \mathcal {C}_R$ let
\begin{equation}\label{act1}
g(\alpha,v)=( \alpha ^g,gv),\, g \in \GL_n(B),\, \alpha \in
\ran(B), \,v \in \mathbb{A}^n_R(B).
\end{equation}
It it is clear that $\uan$ is stable under the above action. Therefore we have that $\uan$ is an open $\GL_n-$subscheme of $\rana$.

\begin{proposition}\label{bala}
For the action described above, $\uan\to\uan/\GL_n$ is a locally-trivial principal $\GL_n-$bundle which is a universal categorical quotient.
\end{proposition}
\begin{proof}
In \cite[Proposition 1]{No} the proposition is proved for the case $A=F$ and $k=R=\ZZ$. It has been extended in \cite[Theorem 7.16]{Ba} for arbitrary $R$. The statement then follows by observing that $\rana$ is a closed $\GL_n-$subscheme of
$\mathrm{Rep}_F^n \times _R \mathbb{A}^n_R$.
\qed\end{proof}

\section{The non commutative Hilbert scheme }\label{nchs}
For $A\in\nn_R$ we recall the definition and the main properties of a
representable functor of points whose representing scheme is the usual Hilbert scheme of points
when $A$ is commutative. For references see for example \cite{Ba,No,Re,Se1,Se2,VdB}.

\begin{definition}
For any algebra $B\in\mathcal{C}_R$ we write
\[
\begin{array}{ll}
\Hilb_A^n(B) := & \{\mbox {left ideals } I
\mbox { in } A \otimes_{R} B \mbox { such that } M=A \otimes _R B /I \
\mbox {is projective} \\
& \mbox{ of rank } n
\mbox { as a } B\mbox{-module} \}.
\end{array}
\]
\end{definition}

\begin{proposition}
The correspondence $\mathcal{C}_R\to Sets$ induced by $B\mapsto
\Hilb_A^n(B)$ gives a covariant functor denoted by
$\Hilb_A^n$.
\end{proposition}
\begin{proof}
Straightforward verification.
\qed\end{proof}
\begin{proposition}\cite[Proposition 2]{VdB}
The contravariant functor $R-Schemes\to Sets$ induced
by $\Hilb_A^n$ is representable by an $R-$scheme denoted by $\mathrm{Hilb}^n_A\,$.
\end{proposition}
\begin{proof}
The functor $\mathrm{Hilb}^n_A$ is a closed subfunctor of the grassmanian
functor.
\qed\end{proof}

\smallskip

Let now $B$ be a commutative $k-$algebra. Consider triples $(\rho,m,M)$
where $M$ is a projective $B-$module of rank $n,\,\rho:A \to
\End_B(M)$ is a $k-$algebra homomorphism such that $\rho(R)\subset B$ and
$\rho(A)(Bm)= M.\,$
\begin{definition}\label{triples}
The triples $(\rho,m,M)$ and $(\rho ',m',M')$ are {\em equivalent}
if there exists a $B-$module isomorphism $\alpha : M \to
M'$ such that $\alpha(m)=m'$ and
$\alpha \rho (a)\alpha ^{-1}=\rho '(a),\,$ for all $a \in A.\,$
\end{definition}

These equivalence classes represent $B-$points of $\Hilb^n_A\,$ as
stated in the following
\begin{lemma}\label{bpoints}
If $B$ is a $k-$algebra, the $B-$points of $\Hilb^n_A$ are in
one-one correspondence with equivalence classes of triples
$(\rho,m,M).$
\end{lemma}
\begin{proof}
Let $I\in \Hilb^n_A(B).\,$ Choose $M=A \otimes _R B /I$ and $\rho :A
\rightarrow \End_B(M)$ given by the composition of the left regular
action of $A$ on itself and the $B-$module homomorphism $A \otimes _R B
\rightarrow M.\,$ Finally let $m=1_M.\,$

On the other hand, let $(\rho,m,M)$ be as in the statement, we
consider the map
\begin{equation}
A \otimes _R B\to M, \;\;\;\;\; a \otimes b \longmapsto
\rho(a)(bm).
\end{equation}
This map is surjective and its kernel $I$ is a $B-$point in
$\Hilb_A^n.\,$
If we consider a triple $(\rho ',m',M')\,$ in the same equivalence class
of $(\rho,m,M),\,\,$ we have that
\[\sum \rho'(a)(b\al(m))=\sum \al\rho(a)(b\al^{-1}\al(m))=\al\sum\rho(a)(bm)\]
so that $I'=I$ and the two triples define the same $B-$point in $\Hilb^n_A.\,$
(See also \cite[Lemma 0.1]{Se1}, and \cite[Lemma 3]{VdB}).
\qed\end{proof}

\smallskip

\begin{remark}\label{ralg}
If the triple $(\rho,m,M)$ represents a $B-$point of $\Hilb^n_A ,\,$
then the homomorphism $\rho$ induces an $R-$algebra structure on $B.\,$
Moreover a homomorphism $\rho '$ where $(\rho ',m',M')$ is in the same
equivalence class of $(\rho,m,M),\,$ induces the same $R-$algebra structure
on $B.\,$
\end{remark}

\subsection{Examples}

\subsubsection{Free algebras}\label{free}
Any $A\in\nn_R$ is a quotient of an opportune free
$R$-algebra $R\{x_{\al}\}$  and in
this case it is easy to see that the scheme $\Hilb_A^n$ is a closed subscheme of
$\Hilb_{R\{x_{\al}\}}^n .\,$

\subsubsection{Van den Bergh's results}
In \cite{VdB} $\Hilb^n_A$ was also defined by M. Van den Bergh in the
framework of Brauer-Severi schemes. He proved that the scheme $\Hilb_F^n$ is
irreducible and smooth of dimension $n+(m-1)n^2 \,$ if $k$ is an
algebraically closed field (see \cite[Theorem 6]{VdB}).

\subsubsection{Commutative case: Hilbert schemes of $n$-points.}
\label{commcase} Let now $R=k$ be an algebraically closed field and let  $A$ be commutative. Let $X=\mathrm{Spec}\,A$, the $k-$points of $\Hilb_A^n$ parameterize zero-dimensional
subschemes $Y\subset X$ of length $n.\,$ It is the simplest case of Hilbert scheme
parameterizing closed subschemes of $X$ with fixed Hilbert polynomial $P,\,$  in this case $P$ is the constant polynomial $n.\,$ The scheme $\Hilb_A^n$ is
usually called the {\em Hilbert scheme of $n-$points on} $X \,$ (see for example \cite{BK,LST}).

\subsection{A principal bundle over the Hilbert scheme}\label{bund}
Recall from Proposition \ref{bala} the principal bundle $\uan\to\uan/\GL_n$.
We have the following result.
\begin{theorem}
The scheme $\uan/\GL_n$ represents $\Hilb_A^n$ and $\uan\to \Hilb_A^n$ is a universal categorical quotient and a $\GL_n-$principal bundle.
\end{theorem}
\begin{proof}
In \cite[Proposition 1]{No} the proposition is proved for the case $A=F$ and $k=R=\ZZ$. It has been extended in \cite[Theorem 7.16]{Ba} for arbitrary $R$. The statement then follows by observing that $\rana$ is a closed $\GL_n-$subscheme of
$\mathrm{Rep}_F^n \times _R \mathbb{A}^n_R$.
\qed\end{proof}

\begin{remark}
Madhav Nori gave a direct proof of the representability of $\Hilb_A^n$ for $k=R=\ZZ$
and $A=F\,$ in \cite{No}. In this case, since $\ZZ$ is PID, the Lemma \ref{bpoints} says that $B-$points in $ \Hilb_F^n$ are represented by equivalence classes of triples $(\rho,v,k^n)$ where
\[
\rho : F \to M_n(k)
\]
is an $n-$dimensional representation of the algebra $F$ over $k$ and $v\in
k^n$ is such that $\rho(F)v=k^n.\,$
It is then easy to see that $\uan\to\uan/\GL_n$ is a principal $\GL_n-$bundle and that the above equivalence classes are in bijection with the $\GL_n-$orbits of $\uan$ (see the proof of Lemma \ref{bpoints}). The isomorphism $\uan/\GL_n\cong \Hilb_A^n$ is then obtained by proving that $\Hilb_A^n$ is $\mathrm{Proj}\,P$ where $P=\oplus _d P_d$ with
\[P_d=\{f\, :\, f(g(a_1,\dots,a_m,v))=(\det g)^d f((a_1,\dots,a_m,v)),\, \mbox{for all}\, g \in \GL_n(k)\}.
\]
\end{remark}
\begin{example}
Let $A=\CC[x,y]$ then $\Hilb_A^n(\CC)$ is described by the above Theorem as
\[\{(X,Y,v)\,:\,X,\,Y\in M_n(\CC),\,XY=YX,\,\CC[X,Y]v=\CC^n\}.\]
This description of the Hilbert scheme of $n-$points of $\CC^2$ is one of the key ingredients of the celebrated Haiman's proof of the $n!$ Theorem \cite{Ha}. It has also been widely used by H.Nakajima \cite{Na}.
\end{example}

\section{Symmetric products and divided powers}\label{poly}
In this part of the paper we introduce a representable functor whose representing scheme will play the same role played by the symmetric product in the classical Hilbert-Chow morphism.

\subsection{Polynomial laws}

We first recall the definition of polynomial laws between
$k$-modules. These are mappings that generalize the usual
polynomial mappings between free $k$-modules. We mostly follow
N.Roby (see {\cite{Ro1,Ro2}}) and we refer the interested reader to
its papers for detailed descriptions and proof.

\begin{definition} Let $M$ and $N$ be two
$k$-modules. A \emph{polynomial law} $\varphi$ from $M$ to $N$ is
a family of mappings $\varphi_{_{A}}:A\otimes_{k} M
\to A\otimes_{k} N$, with $A\in\C_{k}$ such that the
following diagram commutes
\begin{equation}
\xymatrix{
A\otimes_{k}M \ar[d]_{f\otimes id_M} \ar[r]^{\varphi_A}
& A\otimes_{k} N \ar[d]^{f\otimes id_N} \\
B\otimes_{k} M \ar[r]_{\varphi_B}
& B\otimes_{k} N }
\end{equation}
for all $A,\,B\in\C _k$ and all $f\in \C _k(A,B)$.
\end{definition}
\begin{definition}
Let $n\in \NN$, if $\varphi_A(au)=a^n\varphi_A(u)$, for all $a\in
A$, $u\in A\otimes_{k} M$ and all $A\in\C_k$, then $\varphi$
will be said \emph{homogeneous of degree} $n$.
\end{definition}
\begin{definition}
If $M$ and $N$ are two $k$-algebras and
\[
\begin{cases} \varphi_A(xy)&=\varphi_A(x)\varphi_A(y)\\
\varphi_A(1_{A\otimes M})&=1_{A\otimes N}
\end{cases}
\]
for $A\in\C_{k}$ and for all $x,y\in A\otimes_{k} M$, then
$\varphi$ is called \emph{multiplicative}.
\end{definition}
We need the following
\begin{definition}
\begin{itemize}
\item for $S$ a set and any
additive monoid $M$, we denote by $M^{(S)}$ the set of functions
$f:S\rightarrow M$ with finite support.
\item let $\al\in M^{(S)}$,
we denote by $\mid \al \mid$ the (finite) sum $\sum_{s\in S}
\al(s)$,
\end{itemize}
\end{definition}
Let $A$ and $B$ be two $k-$modules and $\varphi:A\rightarrow B$ be a
polynomial law. The following result on polynomial laws is a
restatement of Th\'eor\`eme I.1 of {\cite{Ro1}}.
\begin{theorem}\label{roby} Let $S$ be a set.
\begin{enumerate}
\item Let $L=k[x_s]_{s\in S}$ and let $a_{s}$ be elements of $A$ such
that $a_{s}$ is $0$ except for a finite number of $s\in S$, then
there exist $\varphi_{\xi}((a_{s}))\in B$, with $\xi \in \NN^{(S)}$,
such that:
\begin{equation}\varphi_{_{L}}(\sum_{s\in S} x_s\otimes
a_{s})=\sum_{\xi \in \NN^{(S)}} x^{\xi}\otimes
\varphi_{\xi}((a_{s}))\end{equation}
where $x^{\xi}=\prod_{s\in S}
x_s^{\xi_s}$.
\item Let $R$ be any commutative $k-$algebra and let
$r_s\in R$ for $s\in S$, then:
\begin{equation}\varphi_{_{R}}(\sum_{s\in S}
r_s\otimes a_{s})=\sum_{\xi \in \NN^{(S)}} r^{\xi}\otimes
\varphi_{\xi}((a_{s}))\end{equation}
where $r^{\xi}=\prod_{s\in S}
r_s^{\xi_s}$.
\item If $\varphi$ is homogeneous of degree $n$, then one has
$\varphi_{\xi}((a_{s}))=0$ if $| \xi |$ is
different from $n$. That is:
\begin{equation}\varphi_{_{R}}(\sum_{s\in S}
r_s\otimes a_s)=\sum_{\xi \in \NN^{(S)},\,| \xi |=n}
r^{\xi}\otimes \varphi_{\xi}((a_s))\end{equation}
In particular, if $\varphi$ is
homogeneous of degree $0$ or $1$, then it is constant or linear,
respectively.
\end{enumerate}
\end{theorem}
\begin{remark}\label{coeff}
The above theorem means that a polynomial law $\varphi:A\rightarrow
B$ is completely determined by its coefficients
$\varphi_{\xi}((a_{s}))$, with $\xi \in \NN^{(S)}$.
\end{remark}
\begin{remark} If $A$ is a free $k-$module and $\{a_{t}\,
:\, t\in T\}$ is a basis of $A$, then $\varphi$ is completely
determined by its coefficients $\varphi_{\xi}((a_{t}))$, with $\xi
\in \NN^{(T)}$. If also $B$ is a free $k-$module with basis
$\{b_{u}\, :\, u\in U\}$, then $\varphi_{\xi}((a_{t}))=\sum_{u\in
U}\lambda_{u}(\xi)b_{u}$. Let $a=\sum_{t\in T}\mu_{t}a_{t}\in A$.
Since only a finite number of $\mu_{t}$ and $\lambda_{u}(\xi)$ are
different from zero, the following makes sense:
\begin{eqnarray*}\varphi(a)=\varphi(\sum_{t\in T}\mu_{t}a_{t}) =
\sum_{\xi\in
\NN^{(T)}} \mu^{\xi}\varphi_{\xi}((a_{t}))& = & \sum_{\xi\in
\NN^{(T)}} \mu^{\xi}(\sum_{u\in U}\lambda_{u}(\xi)b_{u})\\ & = &
\sum_{u\in U}(\sum_{\xi\in \NN^{(T)}}\lambda_{u}(\xi)
\mu^{\xi})b_{u}.\end{eqnarray*} Hence, if both $A$ and $B$ are free
$k-$modules, a polynomial law $\varphi:A\rightarrow B$ is simply a
polynomial map.
\end{remark}
\begin{definition}\label{polset}
Let $k$ be a commutative ring.
\begin{itemize}
\item[(1)] For $M,N$ two $k-$modules we set $\PP_k^n(M,N)$ for
the set of homogeneous polynomial laws $M\rightarrow N$ of degree
$n$.
\item[(2)] If $M,N$ are two $k-$algebras we set $\MP_k^n(M,N)$ for
the multiplicative homogeneous polynomial laws $M\rightarrow N$ of
degree $n$.
\end{itemize}
\end{definition}
The assignment $N\to \PP_k^n(M,N)$ (resp.
$N\to \MP^n_k(M,N)$) determines a functor from
$Mod _k$ (resp. $\mathcal {N}_k$) to $Sets$.
\begin{example}\label{expol}
\begin{enumerate}
\item For all $A,B\in\nn_k$ it holds $\MP_k^1(A,B)=\nn_k(A,B)$. A linear
multiplicative polynomial law is a $k-$algebra homomorphism.
\item For $B\in\C_k$ the usual determinant $\mathop{det}:M_n(B)\to B$
belongs to $\MP_k^n(M_n(B),B)$
\item For $B\in\C_k$ the mapping $b\mapsto b^n$ belongs to $\MP_k^n(B,B)$
\item For $B,C\in\C_B$ consider $p\in\MP_B^n(B,C)$, since
$p(b)=p(b\,1)=b^np(1)=b^n\,1_C$ it is clear that in this case raising to
the power $n$ is the unique multiplicative polynomial law homogeneous of
degree $n$.
\item When $A\in\nn_k$ is an Azumaya algebra of rank $n^2$ over its
center $k$, then its reduced norm $N$ belongs to $\MP_k^n(A,k)$.
\end{enumerate}
\end{example}
\subsection{Divided powers}
The functors just introduced in Def. \ref{polset} are represented
by the divided powers which we introduce right now.
\begin{definition}
For a $k$-module
$M$ the {\em divided powers algebra} $\Gamma_{k}(M)$ (see
\cite{Ro1,Ro2}) is an
associative and commutative $k$-algebra with identity $1_{k}$
and product $\times$, with generators $\pd m k $, with $m\in M$,
$k \in \Z$ and relations, for all $m,n\in M$:
\begin{enumerate}
\item $\pd m i = 0, \forall i<0$;
\item $\pd m 0 = 1_{\scriptstyle{k}}, \forall m\in M$;
\item $\pd {(\al m)} i = \al^i\pd m i, \forall \al\in k,
\forall i\in \NN$;
\item $\pd {(m+n)} k = \sum_{i+j=k}\pd m i \pd n j , \forall
k\in \NN$;
\item $\pd m i \times \pd m j = {i+j\choose i}\pd m {i+j} ,
\forall i,j\in \NN$.
\end{enumerate}
\end{definition}
The $k$-module $\Gamma_{k}(M)$ is generated by finite
products $\times_{i\in I}\, \pd {x_i} {\al_i}$ of the above
generators. The divided powers algebra $\Gamma_{k}(M)$ is a
$\NN$-graded algebra with homogeneous components
$\Gamma^n_{k}(M)$, ($n\in \NN$), the submodule generated by
$\{\times_{i\in I} \pd {x_i} {\al_i}\;
:\;\mid\al\mid=\sum_i\al_i=n\}$. One easily checks that
$\Gamma_{k}$ is a functor from $Mod _k$ to $\mathcal C _k$.
\subsection{Universal properties}

The following properties give the motivation for the introduction
of divided powers in our setting.
\subsubsection{Functoriality and adjointness}\label{upgamma}
$\Gamma_{k}^n$ is a covariant functor from $Mod_{k}$ to $Mod_{k}$
and one can easily check that it preserves surjections. The map
$\gamma^n:r\mapsto \pd r n$ is a polynomial law
$M\rightarrow\Gamma^n_{k}(M)$ homogeneous of degree $n$. We call
it \emph{the universal map} for the following reason: consider
another $k-$module $N$ and the set $Mod_{k}(\Gamma^n_{k}(M),N)$ of
homomorphisms of $k-$modules between $\Gamma^n_{k}(M)$ and $N$, we
have an isomorphism
\begin{equation}Mod_{k}(\Gamma^n_{k}(M),N)\xrightarrow{\cong}
\PP^n_{k}(M,N)\end{equation}
given by
$\varphi\mapsto\varphi\circ\gamma^n$.
\subsubsection{The algebra $\Gamma_{k}^n(A)$}
If $A$ is $k-$algebra then $\Gamma^n_{k}(A)$ inherits a
structure of $k-$algebra by
$\pd a n\pd b n=\pd{(ab)}n$.
The unit in $\Gamma_{k}^n(A)$ is $\pd 1 n$. It was proved by
N.Roby {\cite{Ro2}} that in this way $R\to\Gamma_{k}^n(A)$ gives a
functor from $k$-algebras to $k-$algebras such that
$\gamma^n(a)\gamma^n(b)=\gamma^n(ab)$, $\forall a,b \in A$. Hence
\begin{equation}\label{mp}
\nn_{k}(\Gamma^n_{k}(A),B)\xrightarrow{\cong} \MP^n_{k}(A,B)\end{equation}
and the map given by $\varphi\mapsto\varphi\circ\gamma^n$ is an isomorphism
for all $A,B\in
\nn_{k}$.
\cite[Lemma 1.2.]{DPRR}
\subsubsection{Basis change}\label{bc}\cite[Thm. III.3, p. 262]{Ro1}
For any $R\in \C_k$ and $A\in Mod_k$ it holds that
\begin{equation}R\otimes_k\Gamma_k(A)\xrightarrow{\cong}\Gamma_R(R\otimes_k
A).\end{equation}
When $A$ is a $k-$algebra this gives an isomorphism of $R-$algebras
\begin{equation}\label{bcn}
R\otimes_k\Gamma_k^n(A)\xrightarrow{\cong}\Gamma_R^n(R\otimes_k
A)\end{equation}
for all $n\geq 0$.

\subsection{Symmetric tensors}\label{gt}
\begin{definition}\label{symt}
Let $M$ be a $k-$module and consider the $n-$fold tensor power $M^{\otimes n}$.
The symmetric group $S_n$ acts on
$M^{\otimes n}$ by permuting the factors and we denote by $\TS^n_{k}(M)$
or simply by $\TS^n(M)$ the
$k-$submodule of $M^{\otimes n}$ of the invariants for this action. The
elements of $\TS^n(M)$ are called
symmetric tensors of degree $n$ over $M$.
\end{definition}
\begin{remark}
If $M$ is a $k-$algebra then $S_n$ acts on $M^{\otimes n}$ as a group of
$k-$algebra automorphisms. Hence
$\TS^n(M)$ is a $k-$subalgebra of $M^{\otimes n}$.
\end{remark}
\subsubsection{Flatness and Symmetric Tensors}\label{flat}
Suppose $M\in\nn_{k}$ (resp. $M\in\C_{k}$). The homogeneous polynomial law
$M\rightarrow \TS_k^n(M)$ given by $x\mapsto x^{\otimes n}$ gives
a morphism $\tau_n:\Gamma_{k}^n(M)\rightarrow \TS_k^n(M)$ which is
an isomorphism when $M$ is flat over $k$. Indeed one can easily prove
that $\tau_n$ is an isomorphism in case $M$ is free. The flat case then
follows because any flat $k-$module is a direct limit of free modules
and $\Gamma$ (resp. $\Gamma^n$) commutes with direct limits.
\begin{remark}
Divided powers and symmetric tensor are not always isomorphic. See
{\cite{Lu}} for counterexamples.
\end{remark}
\subsection{The abelianization}
In this paragraph we introduce the abelianizator and we prove some
of its properties. Then we deduce the fact that the abelianization
of divided powers commutes with base change.
\begin{definition}
Given a $k-$algebra $A$ we denote by $[A]$ the two-sided ideal of
$A$ generated by the commutators $[a,b]=ab-ba$ with $a,b\in A$. We
write
\[A^{ab}=A/[A]\]
and call it the abelianization of $A$.
\end{definition}
We collect some facts regarding this construction.
\begin{proposition}\label{abefun}
\begin{enumerate}
\item For all $B\in \C_k$ the surjective homomorphism
$\mathfrak{ab}_A:A\to A^{ab}$ gives an isomorphism $\C_k(A^{ab},B)\to
\nn_k(A,B)$ via $\rho\mapsto\rho\cdot\mathfrak{ab}_A$. Equivalently for
all $\varphi\in\nn_k(A,B)$ there is a unique
$\overline{\varphi}:A^{ab}\to B$ such that the following diagram commutes
\[
\xymatrix{
A \ar[dr]_{\varphi} \ar[r]^{\mathfrak{ab}_A}
& A^{ab} \ar[d]^{\overline{\varphi}} \\
& B }
\]
\item The assignment $A\to A^{ab}$ induces a covariant functor
$\nn_k\to\C_k$ that preserves surjections.
\end{enumerate}
\end{proposition}
\begin{proof}
The first point is obvious.
For $A,B\in \nn_k$ and $f\in\nn_k(A,B)$ we have that $f([A])\subset
[B]$. This proves the functoriality.
\qed\end{proof}
\begin{definition}
We call the just introduced functor {\textit{the abelianizator}}.
For $A,B\in \nn_k$ and $f\in\nn_k(A,B)$ we denote by $f^{ab}\in
\C_k(A^{ab},B^{ab})$ the abelianization of $f$.
\end{definition}
\begin{proposition}\label{absu}
The abelianizator preserves surjections.
\end{proposition}
\begin{proof}
Let $A,B\in \nn_k$ and suppose $f\in\nn_k(A,B)$ to be surjective. We
have that $[B]=f([A])$.
\qed\end{proof}
\begin{theorem}\label{abcom}
For all $A\in\nn_k$ and all $R\in\C_k$ it holds that
\[(R\otimes_k A)^{ab}\cong R\otimes_k A^{ab}\]
\end{theorem}
\begin{proof}
The homomorphism $id_R\otimes \mathfrak{ab}_A:R\otimes_k A \to
R\otimes_k A^{ab}$ of $R-$algebras induces a unique one
$\alpha:(R\otimes_kA)^{ab}\to R\otimes_k A^{ab}$ making the following
diagram commutative
\[
\xymatrix{
R\otimes_k A \ar[dr]_{\mathfrak{ab}_{R\otimes A}}
\ar[r]^{id_R\otimes\mathfrak{ab}_A}
& R\otimes_k A^{ab} \\
& (R\otimes_k A)^{ab} \ar[u]^{\alpha} }
\]
On the other hand the homomorphism of $k-$algebras given by the
composition \[A\to R\otimes_k A\xrightarrow{\mathfrak{ab}_{R\otimes
A}}(R\otimes_k A)^{ab}\] induces a unique one $\beta:A^{ab}\to
(R\otimes_k A)^{ab}$. This $\beta$ extends to a homomorphism of
$R-$algebras $\tilde{\beta}:R\otimes_k A^{ab}\to (R\otimes_k A)^{ab}$
that is the inverse of $\alpha$ as can be easily checked.
\qed\end{proof}
\begin{corollary}\label{bcab} For all $n\geq 1$
\[\Gamma_R^n(R\otimes_k A)^{ab}\cong R\otimes_k\Gamma_k^n(A)^{ab}\]
\end{corollary}
\begin{proof}
It follows from (\ref{bcn}) and Proposition~\ref{abcom}.
\qed\end{proof}
\begin{remark}
Corollary \ref{bcab} remain true by replacing $\Gamma^n$ with $\TS^n$ when $A$ is flat.
In particular when $R$ is a field.
\end{remark}
\begin{definition}\label{san}
We denote by $\san$ the affine $R-$scheme Spec $\Gamma_R^n(A)^{ab}$.
\end{definition}
\begin{proposition}\label{sanbed}
Suppose $A=R \otimes F/J \in \nn_r$ then the induced morphism $\san\to \mathrm{S}_{R\otimes F}^n \cong \mathrm{S}_F^n\times_k\mathrm{Spec}\,R$ is a closed immersion.
\end{proposition}
\begin{proof}
It follows from Proposition \ref{absu} and \ref{bc}.
\end{proof}
\begin{remark}
In view of Definition \ref{san} all the results of this section can be rephrased in terms of $R-$schemes. In particular Corollary \ref{bcab} means that $\san$ base-changes well.
\end{remark}
\begin{remark}\label{xn}
When $A$ is flat (see paragraph \ref{flat}) and commutative one sees that
\[\san\cong X^n/S_n=\,\mathrm{Spec}\,\TS^n_R(A)\]
the $n-$fold symmetric product of $X=$Spec $A$.

In particular it follows that $\mathrm{S}_F^n\cong \mathrm{Spec}\,\TS_k^n(F)^{ab}$.
\end{remark}
\section{Morphisms}\label{mor}
In this section we introduce morphisms which connect $\ran//\GL_n,\,
\Hilb^n_A$ and $\san.$

\subsection{The forgetful map}\label{proj}
Recall from Section \ref{bund} that we have $\uan\to\uan/\GL_n\cong\Hilb_A^n$. We have then a commutative diagram
\begin{equation}\label{univ}
\xymatrix{
&\rana\ar@{->>}[rd]&\\
\uan \ar@{^{(}->}[ru] \ar[rr] \ar@{->>}[d]  && \ran\ar@{->>}[d]\\
\Hilb_A^n \ar[rd]\ar[rr]^{p} && \ran//\GL_n \\
& \rana//\GL_n\ar[ru]&}
\end{equation}

\begin{theorem}\label{proje}
 The morphism
\[
 p:\Hilb_A^n \to \ran//\GL_n
\]
in (\ref{univ}) is projective.
\end{theorem}
\begin{proof}
By Theorem 7.16 and Remark 7.17 in \cite{Ba} this is true for $A=F$. The result follows since $\Hilb_A^n$ is a closed subscheme of $\Hilb_F^n$ and $\ran//\GL_n$ is affine.
\qed\end{proof}
The fibers of the map $p$ are very difficult to study. Some results is known for the case $A=F$ and $k$ algebraically closed field \cite{Le}.

\subsection{The norm map}\label{nmap}
Let $\rho\in\nn_R(A,B)$ be a representation of $A$ over $B\in\C_R$. The composition $\det\cdot\rho$ is a multiplicative polynomial law homogeneous of degree $n$ belonging to $MP_R^n(A,B)$. We denote by $\det_{\rho}$ the unique homomorphism in $\C_R(\Gamma_R^n(A)^{ab},B)$ such that $\det\cdot\rho=\det_{\rho}\cdot\gamma^n$, see (\ref{mp}). The correspondence $\rho\mapsto\det_{\rho}$ is clearly functorial in $B$ giving then a morphism $\Gamma_R^n(A)^{ab}\to V_n(A)$ by universality and henceforth a morphism of $R-$schemes
\begin{equation}\label{ddet}
\ddet:\ran\to\san
\end{equation}
Let us look a little bit deeper into the nature of this map.
Let $\bar{\rho}:V_n(A)\to B$ be the unique $B-$point of $\ran$ such that $\rho=M_n(\bar{\rho})\cdot\pi_A$, see Definition \ref{pigreco}.
  It is easy to check that
 \[\bar{\rho}\cdot\det\cdot\pi_A=\det\cdot M_n(\bar{\rho})\cdot\pi_A=\det\cdot\rho\]
 so that the following commutes
 \begin{equation}\label{detro}
\xymatrix{
A \ar[r]^(.3){\pi _A} \ar@{=}[d]&
M_n(V_n(A))\ar[d]^{M_n(\bar{\rho})}\ar[r]^{\det}&V_n(A)\ar[d]^{\bar{\rho}} \\
A\ar[r]^(.3){\rho}& M_n(B)\ar[r]^{\det}&B }
\end{equation}
It follows that $\ddet$ is the affine morphism corresponding to the composition of the universal map $\pi _A$ introduced in ({\ref{pigreco}}) with the determinant i.e. to the top horizontal arrows of  diagram (\ref{detro}).

The determinant is invariant under basis changes and $\ran//\GL_n$ is a categorical quotient. Hence there exists a unique morphism $\dddet:\ran//\GL_n\to\san$ such that the following commutes
\begin{equation}\label{invcnm}
\xymatrix{\ran\ar[rd]^{\ddet}\ar[d]\\
\ran//\GL_n\ar[r]_(.6){\dddet}&\san}
\end{equation}
We have the following result.
\begin{theorem}
The morphism $\dddet : \ran//\GL_n\rightarrow\san$ has the following properties.
\begin{enumerate}
\item Suppose $k=R$ is an infinite field then
    \subitem when $A=k\{x_1,\dots,x_m\}$ then $\dddet$ is an isomorphism,
    \subitem when $A\in\C_k$ then $\dddet$ induces an isomorphism between the associated reduced schemes.
\item Suppose $k=R$ is a characteristic zero field then
    \subitem when $A\in\nn_k$ then $\dddet$ is a closed embedding,
    \subitem when $A\in\C_k$ then $\dddet$ is an isomorphism.
\item Suppose $k=R=\ZZ$ then $\dddet$ is an isomorphism when $A=k\{x_1,\dots,x_m\}$.
\item When $A$ is commutative and flat as $R-$module then $\Gamma_R(A)\cong\TS_R^n(A)\to V_n(A)$ is an injective homomorphism.
\end{enumerate}
\end{theorem}
\begin{proof}
When $A$ is the free $k-$algebra the result follows from Theorem 1.1 in \cite{V1}.
When $A$ is commutative all the properties listed has been proved in \cite{V4,V2,V3}.
The fact that $\dddet$ is a closed embedding in characteristic zero follows considering the following diagram
\[
\xymatrix{
\mathrm{Rep}_F^n//\GL_n \ar[r]_(.6){\cong}^(.6){\dddet} & \mathrm{S}_F^n \\
\ran//\GL_n \ar[r]^(.6){\dddet} \ar[u] & \san \ar[u]}
\]
recalling Proposition \ref{sanbed} and the fact that $\GL_n$ in this case is linear reductive.
\end{proof}

\subsection{Non commutative Hilbert-Chow}
The Hilbert scheme enters the scene again.
\begin{definition}
We denote by $hc$ the morphism given by
\begin{equation}\label{CD1}
\xymatrix{
\Hilb ^n _A \ar@/_1pc/[rr]_{hc} \ar[r]^(.4){p} & \ran//\GL_n \ar[r]^(.6){\dddet}&\san
 }
\end{equation}
\end{definition}

There is the following
\begin{theorem}\label{nforget}
The morphism $hc$ is projective.
\end{theorem}
\begin{proof}
We know $p$ to be projective by Theorem \ref{proje}. Since $\dddet$ is affine and hence separated the result follows.
\qed\end{proof}

\subsubsection{Commutative case:  the Hilbert-Chow morphism and the Grothendieck-Deligne norm.}\label{comm}
In this paragraph $A$ is a commutative $R-$algebra and let $X=\mbox{Spec} \ A.\,$
Let $R=k$ be a an algebraically closed field of arbitrary characteristic.  In this case $A$ is flat over $k\,$ and we have an isomorphism $\san\cong X^n/S_n$, the $n-$fold symmetric product of $X.\,$
There is a natural set-theoretic map
\[
\Hilb^n_A \to X^n/S_n
\]
mapping a zero-dimensional subscheme $Z $  in $\Hilb^n_A$ to the $0-$cycle of degree $n\,$:
\begin{equation}
\displaystyle{ [Z]\, =\,\sum _{P \in |Z|}\, \dim _k(\mathcal O_{Z,P})[P]}.
\end{equation}
This is indeed a morphism of schemes, the \textit {Hilbert-Chow morphism} (see for example \cite{BK,Fo,Iv,Ko} ). If $X$ is a non singular curve, this map is an isomorphism (see \cite{De,Gr}). If $X$ is a non singular surface, the Hilbert-Chow morphism is a resolution of singularities of $X^n/S_n\,$ (see \cite{Fo}), but this is no longer true in higher dimensions.

Suppose now $k$ is again a commutative ring and $R\in\C_R$. The Grothendieck-Deligne norm map was firstly introduced in \cite{LST} as a natural transformation
\begin{equation}
n _A :  \Hilb_A^n  \to \ran
\end{equation}
generalizing previous works of A.Grothendieck \cite{Gr} and P.Deligne 6.3.8 \cite{De}.  It works this way: given a triple $(\rho,m,M)$ representing a $B-$point  of $\Hilb_A^n$ (see Lemma \ref{bpoints}) we can consider the composition $\det\cdot\rho$ and this induces the natural transformation $n_A$.

\begin{remark}
The Grothendieck-Deligne norm map was the starting point and the inspiration for this work.
\end{remark}
\begin{remark}
It should be clear to the reader that in the preceding hypotheses $hc$ is the same as $n_A$.
\end{remark}

\begin{proposition}[ Par.4.5\cite{ES}]
Let $k$ be an algebraically closed field. The Grothendieck-Deligne norm and the Hilbert-Chow morphism coincide on the $k$-points.
\end{proposition}

\begin{theorem}
The Hilbert-Chow morphism is projective
\end{theorem}
\begin{proof}
It follows by Theorem \ref{nforget} and the Proposition above.\qed
\end{proof}

\begin{theorem}\label{final}
The image of the non commutative Hilbert-Chow morphism is isomorphic to the one of the forgetful map in the following cases
\begin{enumerate}
\item when $A$ is the free algebra and $k=R=\ZZ$ or it is an infinite field.
\item when $A$ is commutative and $k=R$ is a characteristic zero field.
\end{enumerate}
\end{theorem}
\begin{proof}
In the listed cases we have that $\dddet$ is an isomorphism hence we have an isomorphism between the image of $hc$ with the one of the forgetful map.
\end{proof}
\begin{example}
Suppose $A=F$ the free algebra. In this case $\ran//\GL_n=M_n^m(k)//\GL_n$ the coarse moduli space of $m-$tuples of $n\times n$ matrices modulo the action of $\GL_n$ by simultaneous conjugation.
The scheme $\mathrm{Hilb}_A^n$ is in this case identified with the quotient by the above action of the open of $\rana$ of the tuples $(a_1,\dots,a_m,v)$ such that $v$ is cyclic i.e. $\{f(a_1,\dots,a_m)v\,:\, f\in A\}$ linearly generates $k^n$. In this case the Hilbert-Chow morphism can be identified with the map induced by $(a_1,\dots,a_m,v)\mapsto (a_1,\dots,a_m)$.

The picture works for $A=k[x_1,\dots,x_m]$ adding the extra condition that $a_ia_j=a_ja_i$ for all $i,\,j$.
\end{example}

\begin{acknowledgements}
The authors would like to thank Michel Brion, Corrado De Concini and Claudio Procesi for fruitful discussions on these topics.
\end{acknowledgements}

\end{document}